\documentclass[12pt,reqno]{amsart}
\usepackage{amscd,amsmath,amsthm,amssymb}
\usepackage{color}
\usepackage{tikz}
\usepackage{url}
\usepackage{circuitikz}
\usepackage{float}
\usepackage{lipsum}

\newtheorem{Theorem}{Theorem}[section]
\newtheorem{Lemma}[Theorem]{Lemma}
\newtheorem{Corollary}[Theorem]{Corollary}
\newtheorem{Proposition}[Theorem]{Proposition}
\newtheorem{Remark}[Theorem]{Remark}

\newtheorem{Example}[Theorem]{Example}

\newtheorem{Question}[Theorem]{Question}

\def\qed{\ifhmode\textqed\fi
	\ifmmode\ifinner\quad\qedsymbol\else\dispqed\fi\fi}
\def\textqed{\unskip\nobreak\penalty50
	\hskip2em\hbox{}\nobreak\hfill\qedsymbol
	\parfillskip=0pt \finalhyphendemerits=0}
\def\dispqed{\rlap{\qquad\qedsymbol}}

\def\m{\mathfrak{m}}
\def\n{\mathfrak{n}}
\def\height{\textup{height}}
\def\Ass{\textup{Ass}}

\def\Tor{\textup{Tor}}
\def\depth{\textup{depth\,}}
\def\lcm{\textup{lcm}}

\def\pd{\textup{proj\,dim}}
\def\supp{\textup{supp}}
\def\reg{\textup{reg}}

\def\lex{\textup{lex}}

\def\girth{\textup{girth}}

\begin{document}
	
	\title{Complementary edge ideals}
	\author{Antonino Ficarra, Somayeh Moradi}
	
	\address{Antonino Ficarra, BCAM -- Basque Center for Applied Mathematics, Mazarredo 14, 48009 Bilbao, Basque Country -- Spain, Ikerbasque, Basque Foundation for Science, Plaza Euskadi 5, 48009 Bilbao, Basque Country -- Spain}
	\email{aficarra@bcamath.org,\,\,\,\,\,\,\,\,\,\,\,\,\,antficarra@unime.it}
	
	\address{Somayeh Moradi, Department of Mathematics, Faculty of Science, Ilam University, P.O.Box 69315-516, Ilam, Iran}
	\email{so.moradi@ilam.ac.ir}
	
	
	\begin{abstract}
		Let $S=K[x_1,\dots,x_n]$ be the polynomial ring over a field $K$ and $I\subset S$ be a squarefree monomial ideal generated in degree $n-2$.  Motivated by the remarkable behavior of the powers of $I$ when 
		$I$ admits a linear resolution, as established in \cite{F25b}, in this work we investigate the algebraic and homological properties of 
		$I$ and its powers.
		To this end, we introduce the complementary edge ideal of a finite simple graph $G$ as the ideal   $$I_c(G)=((x_1\cdots x_n)/(x_ix_j):\{i,j\}\in E(G)) $$
		of $S$, where $V(G)=\{1,\ldots,n\}$ and $E(G)$ is the edge set of $G$. By interpreting any squarefree monomial ideal $I$ generated in degree $n-2$ as the complementary edge ideal of a graph $G$, we establish a correspondence between algebraic invariants of $I$ and combinatorial properties of $G$. More precisely, we characterize sequentially Cohen-Macaulay, Cohen-Macaulay, Gorenstein, nearly Gorenstein and matroidal complementary edge ideals. Moreover, we determine the regularity of powers of $I$ in terms of combinatorial invariants of the graph $G$ and obtain that $I^k$ has linear resolution or linear quotients for some $k$ (equivalently for all $k\geq 1$) if and only if $G$ has only one connected component with at least two vertices.
	\end{abstract}
	
	\maketitle\vspace*{-1.4em}
	\section*{Introduction}
	Let $S = K[x_1, \dots, x_n]$ be the standard graded polynomial ring over a field $K$, and let $I \subset S$ be a squarefree monomial ideal. Characterizing when $I$ has a linear resolution is a classical and central problem in Commutative Algebra. Let $\mathcal{I}_{n,d}(K)$ denote the class of squarefree monomial ideals in $S$ having a $d$-linear resolution.
	
	In our previous work \cite{F25b}, we proved that the class $\mathcal{I}_{n,d}(K)$ is independent of the ground field $K$ if and only if $d\in\{0,1,2,n{-}2,n{-}1,n\}$. That is, a uniform combinatorial description of the linearity of the resolution exists for all squarefree monomial ideals of degree $d$ only in these six exceptional cases. Among them, the cases $d=0,1,n{-}1,n$ are trivial: every ideal generated in those degrees has linear resolution. The case $d=2$ was addressed by Fröberg in his seminal work \cite{Fr}, and Herzog, Hibi, and Zheng later established in \cite{HHZ} that, if $I$ has a $2$-linear resolution, then all powers $I^k$ have $2k$-linear resolutions for every $k\ge1$.
	
	In \cite{F25b}, we extended this phenomenon to all degrees $d\in\{0,1,2,n{-}2,n{-}1,n\}$, showing that the corresponding ideals in $\mathcal{I}_{n,d}(K)$ have powers with linear resolutions. Our main contribution in that work was the treatment of the case $d=n{-}2$, which we identified as, in a precise sense, complementary to the classical case $d=2$.
	
	To formalize this idea, we introduce the following notion. Given a non-zero squarefree monomial ideal $I\subset S$ with unique minimal monomial generating set $\mathcal{G}(I)=\{u_1,\dots,u_m\}$, we define its \textit{complementary ideal} as the ideal
	\[
	I_c(I)\ =\ ((x_1\cdots x_n)/u_1,\dots,(x_1\cdots x_n)/u_m).
	\]
	
	If $I=(0)$, we put $I_c(I)=(0)$. This construction defines an involution on the set of squarefree monomial ideals and plays a crucial role in our framework.
	
	Now, let $G$ be a finite simple graph on the vertex set $V(G)=[n]=\{1,2,\dots,n\}$. The \emph{edge ideal} $I(G)$ of $G$ is the squarefree monomial ideal of $S$ generated by the monomials $x_ix_j$ for each edge $\{i,j\}\in E(G)$. Edge ideals have been extensively studied, and their theory is well-developed \cite{HHBook,RV}.
	
	Any squarefree monomial ideal $I \subset S$ generated in degree $d=n{-}2$ is the complementary ideal of an edge ideal. In \cite[Theorem B]{F25b}, we showed that $I\in\mathcal{I}_{n,n{-}2}$ if and only if $I=I_c(I(G))$ for some graph $G$ which has only one connected component with more than one vertex. Moreover, in this case, we proved that $I^k$ has not only linear resolution, but also linear quotients, for all $k\ge1$, complementing the celebrated result of Herzog, Hibi and Zheng \cite{HHZ}.
	
	In the present work, we further explore this connection. Given a finite simple graph $G$, we introduce the \emph{complementary edge ideal} of $G$, as the ideal  
	\[
	I_c(G)\ =\ ((x_1\cdots x_n)/(x_ix_j)\ :\ \{i,j\}\in E(G)) 
	\]
	of $S$.
	These ideals naturally appear as the Alexander duals of the Stanley-Reisner ideals of pure one-dimensional simplicial complexes. We denote the Alexander dual $I_c(G)^\vee$ of $I_c(G)$ by $J_c(G)$, and refer to it as the \emph{complementary cover ideal} of $G$.
	
	While Stanley-Reisner ideals of (pure) one-dimensional simplicial complexes have received some attention, primarily from topological and combinatorial perspectives (see \cite{HT,MV23}), little is known about the ideals $I_c(G)$ and $J_c(G)$ when viewed directly in terms of the underlying graph $G$. This paper initiates a systematic study of these ideals, focusing on their algebraic and homological features. In particular, we determine their minimal primary decompositions, characterize when they are (sequentially) Cohen-Macaulay or (nearly) Gorenstein, and we compute explicitly the Castelnuovo-Mumford regularity of each of their powers.
	
	Our paper is structured as follows. In Section \ref{sec1} we study general properties of complementary ideals. In Proposition \ref{prop:I_c-properties}, we notice that $I_c(I_c(I))=I$, that is the operation $I_c(\mkern1mu\underline{\hspace{0.2cm}}\mkern1mu)$ defines an involution on the set of squarefree monomial ideals of $S$. In general, the properties of having linear quotients or having linear resolution are not preserved by taking the complementary ideal (Example \ref{ex:Ic}). In Theorem \ref{thm:Icpolym} we prove that $I_c(\mkern1mu\underline{\hspace{0.2cm}}\mkern1mu)$ preserves the matroidal property, and hence the linear quotients property for the family of matroidal ideals.
	
	In Section \ref{sec2}, we introduce complementary edge ideals. Notice that, if $G$ contains an isolated vertex, say $n$, and if we let $H$ to be induced subgraph of $G$ on the vertex set $[n-1]$, then $I_c(G)=x_n I_c(H)$ where
	\begin{equation}\label{eq:IcH}
		I_c(H)=((x_1\cdots x_{n-1})/(x_ix_j)\ :\ \{i,j\}\in E(H)).
	\end{equation}
	
	It is easy to recover the algebraic and homological properties of $I_c(G)$ in terms of $I_c(H)$. For instance, in the case above, if $I_c(H)=\bigcap_{i=1}^tP_i$ is the minimal primary decomposition of $I_c(H)$, then $I_c(G)=(x_n)\cap(\bigcap_{i=1}^tP_i)$ is the minimal primary decomposition of $I_c(G)$. Therefore, for the whole Section \ref{sec2} we assume that $G$ does not contain isolated vertices. In Theorem \ref{thm:I_c(G)pd}, for such a graph $G$, we prove that:
	\[
	J_c(G)\ =\ I(G^c) + K_3(G),
	\]
	where $G^c$ is the complementary graph of $G$, and $K_3(G)$ denotes the $3$-clique ideal generated by all triangles of $G$ \cite{Mor}. From this identity we see that the equality $I_c(G)^\vee = I(G^c)$ holds if and only if $G$ is $K_3$-free (Corollary \ref{cor:Ic(G)unmixed}).
	
	Further building on such relation between $J_c(G)$ and $I_c(G)$, we characterize the sequentially Cohen-Macaulay complementary edge ideals. In Theorem \ref{thm:Ic(G)sCM} we prove that $I_c(G)$ is sequentially Cohen-Macaulay if and only if $G$ is a chordal graph. In this case, by Alexander duality, $J_c(G)$ is componentwise linear. In fact, we even prove that $J_c(G)$ has linear quotients. When $G$ is chordal, due to some experimental evidence, we expect that $J_c(G)^k$ has linear quotients for all $k\ge1$ (Question \ref{que:Jc}). In Theorem~\ref{thm:Ic(G)CM} we establish that $I_c(G)$ is Cohen-Macaulay if and only if $G$ is either the complete graph on $n$ vertices or a forest. In Theorem \ref{thm:IcGor} we prove that $I_c(G)$ is Gorenstein, if and only if, $G\in\{K_2,\,K_3,\,2K_2,\,P_3\}$. Moreover, in Corollary \ref{cor:Ic(G)Gor} we prove that $I_c(G)$ is nearly Gorenstein, but not Gorenstein, if and only if, $G\in\{K_4,\,P_4\}$. Here $K_s$ denotes a complete graph on $s$ vertices, $P_s$ denotes a path on $s$ vertices, and $2K_2$ denotes the disjoint union of two edges.
	
	In Section \ref{sec3} we recover and strengthen a result due, independently, to Blum \cite{B}, Ohsugi and Hibi \cite{OH}, Nasernejad, Khashyarmanesh and Qureshi \cite{KNQ}. Using Theorem \ref{thm:Icpolym}, in Theorem \ref{thm:Ic(G)polym} we prove that $I(G)$ is matroidal, if and only if, $I_c(G)$ is matroidal, if and only if, $G$ is a complete multipartite graph.
	
	Finally, in the last section we compute explicit regularity functions for both $I_c(G)$ and $J_c(G)$. The case of complementary cover ideals was already addressed by Hoa and Trung \cite[Theorem 2.9]{HT}, and by Minh and Vu \cite[Theorem 1.1]{MV23}. In that case, the authors even computed the regularity of each symbolic power of $J_c(G)$ and showed that $\reg\,J_c(G)^{(k)}=\reg\,J_c(G)^k$ for all $k\ge1$.
	
	Now, let $c(G)$ denote the number of connected components of $G$ which are not isolated vertices. For arbitrary $G$, in Theorem \ref{thm:Ic(G)regdepth} we prove that
	$$
	\reg\,I_c(G)^k\ =\
	\begin{cases}
		(|V(G)| - 1)k & \text{if } k \leq c(G) - 2, \\
		(|V(G)| - 2)k + c(G) - 1 & \text{if } k \geq c(G) - 1.
	\end{cases}
	$$
	
	We also show that the depth function $k\mapsto\depth\,S/I_c(G)^k$ is non-increasing. The proof of these results heavily depends on the computation of explicit Betti splittings of $I_c(G)^k$ and also on \cite[Theorem B]{F25b} where the function $k\mapsto\reg\,I_c(G)^k$ is computed in the case that $c(G)=1$. As a nice consequence, in Corollary \ref{cor:Ic(G)Betti} we show that the graded Betti numbers of $I_c(G)^k$ do not depend on the ground field $K$, for all $k\ge1$. Whether the functions $k\mapsto\!\depth\,S/J_c(G)^k$ and $k\mapsto\!\depth\,S/J_c(G)^{(k)}$ are non-increasing as well, or whether we have $\reg\,I_c(G)^k=\reg\,I_c(G)^{(k)}$ for all $k\ge1$, remain open questions at the moment.
	
	Complementary edge ideals have been introduced independently in \cite{HQS}. Section 2 of this work overlaps with Section 1 of the paper \cite{HQS}.
	
	\section{Complementary ideals}\label{sec1}
	
	Let $I\subset S$ be a monomial ideal with $\mathcal{G}(I)$ its minimal monomial generating set. The \textit{support} of a monomial $u\in S$ is defined as the set $\supp(u)=\{i:\ x_i\ \textup{divides}\ u\}$. Whereas, the \textit{support} of $I$ is defined as $\supp(I)=\bigcup_{u\in\mathcal{G}(I)}\supp(u)$.
	
	Given $F\subset[n]$ non-empty, we put ${\bf x}_F=\prod_{i\in F}x_i$. Moreover, we put ${\bf x}_\emptyset=1$.\medskip
	
	The \textit{complementary ideal} of a non-zero squarefree monomial ideal $I\subset S$ is defined as the squarefree monomial ideal
	$$
	I_c(I)\ =\ ({\bf x}_{[n]}/u\ :\ u\in\mathcal{G}(I)).
	$$
	
	If $I=(0)$, we put $I_c(I)=(0)$.
	More generally, in \cite{ALS} the concept of generalized Newton complementary dual ideals were introduced which associates the complementary dual ideal to any monomial ideal, in order to study  cellular free resolutions of duals of some families of monomial ideals. 
	
	Notice that ${\bf x}_F$ divides ${\bf x}_G$ if and only if ${\bf x}_{[n]\setminus G}$ divides ${\bf x}_{[n]\setminus F}$. As a consequence, the following properties are easily verified.
	\begin{Proposition}\label{prop:I_c-properties}
		Let $I,J\subset S$ be non-zero squarefree monomial ideals. The following properties hold.
		\begin{enumerate}
			\item[\textup{(a)}] $\mathcal{G}(I_c(I))=\{{\bf x}_{[n]}/u\ :\ u\in\mathcal{G}(I)\}$.
			\item[\textup{(b)}] $I_c(I_c(I))=I$.
		\end{enumerate}
	\end{Proposition}
	
	We say that a homogeneous ideal $I\subset S$ has \textit{linear resolution} if $I$ is generated in a single degree $d$ and the $i$th syzygy module of $I$ is either zero or generated in degree $d+i$, for all $i>0$. 
	
	Now, let $I\subset S$ be a monomial ideal. We say that $I$ has \textit{linear quotients} if there exists an order $u_1,\ldots,u_m$ of the set $\mathcal{G}(I)$ such that $(u_1,\ldots,u_{i-1}):(u_i)$ is generated by variables, for all $i = 2,\ldots,m$. In the following, for two monomials $u$ and $v$, we set $u:v=u/\gcd(u,v)=\lcm(u,v)/v$. Notice that
	$$(u_1,\ldots,u_{i-1}):(u_i)=(u_1:u_i,\dots,u_{i-1}:u_i).$$
	
	From this equation, it follows that $I$ has linear quotients order $u_1,\dots,u_m$, if and only if for all generators $u_j:u_i$ with $j<i$ there exists $h<i$ such that $u_h:u_i=x_p$ is a variable that divides $u_j:u_i$, see \cite[Corollary 8.2.4]{HHBook}.
	
	If a monomial ideal is generated in a single degree and it has linear quotients, then it has linear resolution, see \cite[Theorem 8.2.15]{HHBook}.\smallskip
	
	The next example shows that the properties of having linear resolution or linear quotients are not preserved by taking the complementary ideal.
	\begin{Example}\label{ex:Ic}
		\rm Let $I=(x_1x_2x_3,x_1x_2x_5,x_1x_4x_5,x_3x_4x_5)\subset S=K[x_1,x_2,x_3,x_4,x_5]$. Then $I$ has linear quotients, and thus a linear resolution. However, the complementary ideal $I_c(I)=(x_1x_2,x_2x_3,x_3x_4,x_4x_5)$ does not have linear resolution.
	\end{Example}
	
	On the other hand, there is a family of squarefree monomial ideals $I$ whose linear quotients property holds for $I$ if and only if it holds for the complementary ideal.
	
	The \textit{$x_i$-degree} of a monomial $u\in S$ is defined as
	$$
	\deg_{x_i}(u)=\max\{j:\ x_i^j\ \textup{divides}\ u\}.
	$$
	
	A monomial ideal $I\subset S$ is called \textit{polymatroidal} if $I$ is equigenerated and satisfies the \textit{exchange property}: for all $u,v\in\mathcal{G}(I)$ and all $i$ such that $\deg_{x_i}(u)>\deg_{x_i}(v)$, there exists $j$ with $\deg_{x_j}(u)<\deg_{x_j}(v)$ such that $x_j(u/x_i)\in\mathcal{G}(I)$. A squarefree polymatroidal ideal is called \textit{matroidal}. Polymatroidal ideals have linear quotients, see \cite[Theorem 12.6.2]{HHBook}.
	
	By \cite[Lemma 2.1]{HH2003}, a polymatroidal ideal $I$ also satisfy the so-called \textit{dual exchange property}: for all $u,v\in\mathcal{G}(I)$ and all $i$ such that $\deg_{x_i}(u)<\deg_{x_i}(v)$, there exists $j$ with $\deg_{x_j}(u)>\deg_{x_j}(v)$ such that $x_i(u/x_j)\in\mathcal{G}(I)$.
	
	\begin{Theorem}\label{thm:Icpolym}
		Let $I\subset S$ be a squarefree monomial ideal. Then $I$ is matroidal if and only if $I_c(I)$ is matroidal.
	\end{Theorem}
	\begin{proof}
		Suppose that $I$ is matroidal, and let $u,v\in\mathcal{G}(I_c(I))$ with $\deg_{x_i}(u)>\deg_{x_i}(v)$. Then $u={\bf x}_{[n]\setminus F}$ and $v={\bf x}_{[n]\setminus G}$ for some ${\bf x}_F,{\bf x}_G\in\mathcal{G}(I)$. Since $I$ is matroidal and $\deg_{x_i}({\bf x}_F)<\deg_{x_i}({\bf x}_G)$, by the dual exchange property there exists an integer $j$ with $\deg_{x_j}({\bf x}_F)>\deg_{x_j}({\bf x}_G)$ such that $x_i({\bf x}_F/x_j)\in\mathcal{G}(I)$. Then $\deg_{x_j}(u)<\deg_{x_j}(v)$ and
		$$
		{\bf x}_{[n]}/(x_i({\bf x}_F/x_j))\ =\ x_j({\bf x}_{[n]\setminus F}/x_i)\ =\ x_j(u/x_i)\in\mathcal{G}(I_c(I)),
		$$
		as desired. The converse follows from Proposition \ref{prop:I_c-properties}(b) and the `only if' part of the statement. 
	\end{proof}
	
	\section{Complementary edge ideals}\label{sec2}
	
	Let $G$ be a finite simple graph on the vertex set $V(G)=[n]=\{1,2,\dots,n\}$, where $n\ge1$ is a positive integer, and with edge set $E(G)$. The \textit{edge ideal} of $G$ is the squarefree monomial ideal of $S=K[x_1,\dots,x_n]$ defined as follows
	$$
	I(G)\ =\ (x_ix_j\ :\ \{i,j\}\in E(G)).
	$$
	
	For a non-empty subset $F\subset[n]$, we put $P_F=(x_i:\ i\in F)$. Let $I=\bigcap_{i=1}^mP_{F_i}$ be the minimal primary decomposition of a squarefree monomial ideal $I\subset S$. The \textit{Alexander dual} of $I$ is defined as the squarefree monomial ideal $$I^\vee\ =\ ({\bf x}_{F_1},\dots,{\bf x}_{F_m}).$$
	
	It is straightforward to see that $(I^\vee)^\vee=I$. The \textit{cover ideal} of $G$, denoted by $J(G)$, is defined as the Alexander dual of $I(G)$.\medskip
	
	In analogy to the classical definitions, we introduce the \textit{complementary edge ideal} of $G$, denoted by $I_c(G)$, which is defined as the complementary ideal of $I(G)$,
	$$
	I_c(G)\ =\ I_c(I(G))\ =\ ({\bf x}_{[n]}/(x_ix_j)\ :\ \{i,j\}\in E(G)).
	$$
	
	Analogously, we define the \textit{complementary cover ideal} of $G$, and denote it by $J_c(G)$, as the Alexander dual of $I_c(G)$,
	$$
	J_c(G)\ =\ I_c(G)^\vee\ =\ \bigcap_{\{i,j\}\in E(G)}P_{[n]\setminus\{i,j\}}.
	$$\smallskip
	
	To state our results we need some terminology from graph theory. Let $A\subset V(G)$ be non empty. The induced subgraph of $G$ on $A$ is defined as the graph $G_A$ with vertex set $A$ and $\{i,j\}\in E(G_A)$ if and only if $\{i,j\}\in E(G)$ and $i,j\in A$.\smallskip
	
	Let $G,H$ be finite simple graphs. We say that $G$ is \textit{$H$-free} if $G$ contains no induced subgraph isomorphic to $H$. Given an integer $t$, we define $tH$ to be the graph consisting of $t$ disjoint copies of $H$.\medskip
	
	The following three graphs will play a crucial role.
	\begin{enumerate}
		\item[(i)] The complete graph on $n$ vertices is the graph $K_n$ with $V(K_n)=[n]$ and edge set $E(K_n)=\{\{i,j\}:\ 1\le i<j\le n\}$.\smallskip
		\item[(ii)] The path on $n$ vertices is the graph $P_n$ with $V(P_n)=[n]$ and edge set $E(P_n)=\{\{i,i+1\}:\ 1\le i<n\}=\{\{1,2\},\{2,3\},\dots,\{n-1,n\}\}$.\smallskip
		\item[(iii)] The cycle on $n$ vertices is the graph $C_n$ with $V(C_n)=[n]$ and edge set $E(C_n)=\{\{1,2\},\{2,3\},\dots,\{n-1,n\},\{n,1\}\}$.\smallskip
	\end{enumerate}
	
	A graph $G$ is called \textit{chordal} if $G$ is $C_s$-free for all $s\ge4$. Whereas, $G$ is called a \textit{forest} if $G$ is $C_s$-free for all $s\ge3$. Hence, any forest is a chordal graph.\smallskip
	
	A $t$-clique of $G$ is a subset $A\subset V(G)$ of size $t$, whose induced subgraph $G_A$ is isomorphic to $K_t$. A 3-clique of $G$ is called a \textit{triangle} of $G$. We denote by $\mathcal{T}(G)$ the set of all triangles of $G$.\smallskip
	
	Recall that the complementary graph of $G$ is the graph $G^c$ with the same vertex set of $G$ and with the edge set $E(G^c)=\{\{i,j\}:\ 1\le i<j\le n\}\setminus E(G)$.\smallskip
	
	As explained in the introduction, it is not restrictive to assume that $G$ does not contain isolated vertices. We maintain this assumption thorough this section.
	
	Firstly, we compute the minimal primary decomposition of $I_c(G)$.
	\begin{Theorem}\label{thm:I_c(G)pd}
		Let $G$ be a finite simple graph without isolated vertices. Then
		$$
		I_c(G)\ =\ (\bigcap_{e\in E(G^c)}P_e)\cap(\bigcap_{T\in\mathcal{T}(G)}P_T).
		$$
	\end{Theorem}
	\begin{proof}
		Let $P=P_F\in\Ass(I_c(G))$. We claim that $|F|\ge2$. Indeed, if $|F|=1$, say $F=\{i\}$, then $P_{\{i\}}$ is not an associated prime of $I_c(G)$ because ${\bf x}_{[n]\setminus e}\in I_c(G)\setminus P_{\{i\}}$ for any edge $e\in E(G)$ with $i\in e$. Hence $|F|\ge2$ and so $\height\,I_c(G)\ge2$.
		
		Assume now that $|F|=2$. If $F\in E(G)$ then $P_F$ is not an associated prime of $I_c(G)$ because ${\bf x}_{[n]\setminus F}\in I_c(G)\setminus P_{F}$. Hence $F\in E(G^c)$. Conversely, let $e\in E(G^c)$. We claim that $P_e\in\Ass(I_c(G))$. Let $f\in E(G)$. Then $f\ne e$ because $e\notin E(G)$. Hence ${\bf x}_{[n]\setminus f}\in P_e$. This shows that $I_c(G)\subset P_e$. Since $\height\,I_c(G)\ge2$, it follows that $P_e\in\Ass(I_c(G))$.
		
		Finally, suppose that $|F|\ge3$. We claim that $|F|=3$ and $F$ is a triangle in $G$. Let $F=\{i,j,k,\dots\}$. If $e=\{i,j\}\in E(G^c)$ then $I_c(G)\subset P_e\subset P_F$. Since $I_c(G)$ is a radical ideal, it follows that $P_F\notin\Ass(I_c(G))$, a contradiction. Therefore, $T=\{i,j,k\}$ is a triangle in $G$. We claim that $F=T$. Let $f\in E(G)$. We can find an edge of the triangle $T$, say $\{i,j\}$, which is different from $f$. This then shows that ${\bf x}_{[n]\setminus f}\in P_T$. Hence $I_c(G)\subset P_T\subset P_F$. Since $I_c(G)$ is radical, it follows that $F=T$. Conversely, let $T=\{i,j,k\}\in\mathcal{T}(G)$. We claim that $P_T\in\Ass(I_c(G))$. The previous argument shows that $I_c(G)\subset P_T$. If $P_T$ was not an associated prime of $I_c(G)$, since $\height\,I_c(G)\ge2$, then at least one of the three prime ideals $P_{\{i,j\}}$, $P_{\{i,k\}}$, $P_{\{j,k\}}$ should be an associated prime of $I_c(G)$. Say, for instance, $P_{\{i,j\}}$. However, this is not possible because ${\bf x}_{[n]\setminus\{i,j\}}\in I_c(G)\setminus P_{\{i,j\}}$. Hence $P_T\in\Ass(I_c(G))$.
	\end{proof}
	
	As an immediate consequence of this result we have the following corollaries. Following \cite{Mor} (see also \cite{KM}), we define the \textit{$t$-clique ideal} of $G$ as the ideal
	$$
	K_t(G)\ =\ ({\bf x}_F\ :\ F\subset[n],\ |F|=t,\ F\ \textup{is a}\ t\textup{-clique of}\ G).
	$$
	
	\begin{Corollary}\label{cor:J_c(G)dual}
		Let $G$ be a finite simple graph without isolated vertices. Then $$J_c(G)\ =\ I(G^c)+K_3(G).$$
	\end{Corollary}
	
	\begin{Corollary}\label{cor:Ic(G)dim}
		Let $G$ be a finite simple graph without isolated vertices. Then
		$$
		\dim\,S/I_c(G)\ =\ \begin{cases}
			\,n-3&\text{if}\ G\ \text{is the complete graph},\\
			\,n-2&\text{otherwise}.
		\end{cases}
		$$
	\end{Corollary}
	
	\begin{Corollary}\label{cor:Ic(G)unmixed}
		Suppose that $G$ does not have isolated vertices and is not a complete graph. Then, the following conditions are equivalent.
		\begin{enumerate}
			\item[\textup{(a)}] $I_c(G)$ is unmixed.
			\item[\textup{(b)}] $I_c(G)^\vee=I(G^c)$.
			\item[\textup{(c)}] $G$ is $K_3$-free.
		\end{enumerate}
	\end{Corollary}\smallskip
	
	Our next goal is to characterize when $I_c(G)$ is sequentially Cohen-Macaulay.\smallskip
	
	Recall that a graded $S$-module $M$ is called \textit{sequentially Cohen-Macaulay} if there exists a finite filtration of graded $S$-modules
	$$
	M_0\subset M_1 \subset \cdots \subset M_r=M
	$$
	such that each $M_i/M_{i-1}$ is Cohen-Macaulay, and the Krull dimensions of the quotients are increasing: $\dim(M_1/M_0)<\dim(M_2/M_1)<\cdots<\dim(M_r/M_{r-1})$.\smallskip
	
	We say that an ideal $I\subset S$ is (sequentially) Cohen-Macaulay if $S/I$ is (sequentially) Cohen-Macaulay.\smallskip
	
	By \cite[Theorem 8.2.20]{HHBook}, a squarefree monomial ideal $I\subset S$ is sequentially Cohen-Macaulay if and only if $I^\vee$ \textit{componentwise linear}, which means that $I_{\langle j\rangle}$ has linear resolution for all $j\ge0$. Here, for a homogeneous ideal $I\subset S$ and an integer $j\ge0$, $I_{\langle j\rangle}$ denotes the graded ideal of $S$ generated by the homogeneous polynomials of degree $j$ belonging to $I$.\smallskip
	
	Recall that the \textit{squarefree Veronese ideal of degree $d$} in $S$ is the squarefree monomial ideal generated by all squarefree monomials of $S$ having degree $d$,
	$$
	I_{n,d}=(x_{i_1}\cdots x_{i_d}\ :\ 1\le i_1<\dots<i_d\le n).
	$$
	
	Notice that $I_{n,0}=(1)=S$, $I_{n,n}=(x_1x_2\cdots x_n)$ and $I_c(K_n)=I_{n,n-2}$. Furthermore, it is well-known that $I_{n,d}$ has linear quotients.\smallskip
	
	Let $G$ be a finite simple graph. The \textit{open neighbourhood} of $i\in V(G)$ is the set
	$$
	N_G(i)\ =\ \{j\in V(G):\ \{i,j\}\in E(G)\}.
	$$
	
	A \textit{perfect elimination order} of $G$ is an ordering $v_1,\dots,v_n$ of its vertex set $V(G)$ such that $N_{G_i}(v_i)$ induces a complete subgraph of $G_i$, where $G_i$ is the induced subgraph of $G$ on the vertex set $\{v_i,v_{i+1},\dots,v_n\}$. By Dirac \cite{Dirac}, a finite simple graph $G$ is chordal if and only if $G$ admits a perfect elimination order.
	
	\begin{Theorem}\label{thm:Ic(G)sCM}
		Let $G$ be a finite simple graph without isolated vertices. The following conditions are equivalent.
		\begin{enumerate}
			\item[\textup{(a)}] $I_c(G)$ is sequentially Cohen-Macaulay.
			\item[\textup{(b)}] $J_c(G)$ is componentwise linear.
			\item[\textup{(c)}] $J_c(G)$ has linear quotients.
			\item[\textup{(d)}] $G$ is chordal.
		\end{enumerate}
	\end{Theorem}
	\begin{proof}
		The equivalence (a) $\Leftrightarrow$ (b) follows by \cite[Theorem 8.2.20]{HHBook}. Next, we prove the implications (b) $\Rightarrow$ (d), (d) $\Rightarrow$ (c) and (c) $\Rightarrow$ (b).\smallskip
		
		(b) $\Rightarrow$ (d) If $I(G^c)=(0)$, then $G$ is the complete graph and so it is chordal. Otherwise, if $I(G^c)\ne(0)$ then $J_c(G)_{\langle2\rangle}=I(G^c)$ has linear resolution by the assumption. Fr\"oberg theorem \cite[Theorem 9.2.3]{HHBook} implies that $(G^c)^c=G$ is chordal.\smallskip
		
		(d) $\Rightarrow$ (c) If $I(G^c)=(0)$, then $G$ is the complete graph $K_n$, and consequently $J_c(G)=J_c(K_n)=I_{n,3}$ has linear quotients, as desired.
		
		Suppose now that $I(G^c)\ne(0)$, then $G$ is not the complete graph on $n$ vertices. Since $G$ is chordal, \cite[Theorem 10.2.6]{HHBook} implies that $I(G^c)$ has linear quotients. Up to a relabeling, we can assume that $1,2,\dots,n$ is a perfect elimination order of $G$. Let $H=G_{[n]\setminus\{1\}}$. Then, $H$ is again chordal and
		$$
		J_c(G)\ =\ x_1P+I(H^c)+x_1Q+K_3(H),
		$$
		where $P=(x_j:\ x_1x_j\in I(G^c))$ and $Q=(x_rx_s:\ x_1x_rx_s\in K_3(G))$.\smallskip
		
		By \cite[proof of Corollary 2.5]{FShort}, $I(G^c)=x_1P+I(H^c)$ has linear quotients order $u_1,\dots,u_m$ with respect to the lexicographic order $>_{\lex}$ induced by $x_1>\dots>x_n$, and by \cite[Lemma 2.4]{FShort} we have $I(H^c)\subset P$. Let $v_1,\dots,v_\ell$ be the minimal monomial generators of $K_3(G)=x_1Q+K_3(H)$ ordered with respect to the lexicographic order induced by $x_1>\dots>x_n$. We claim that
		$$
		u_1,\dots,u_m,v_1,\dots,v_\ell
		$$
		is a linear quotients order of $J_c(G)$. By what said above, we only need to show that $(u_1,\dots,u_m,v_1,\dots,v_{i-1}):v_i$ is generated by variables for all $i=1,\dots,\ell$. Notice that $J_c(H)=I(H^c)+K_3(H)$. Since $H$ is chordal on $n-1$ vertices with perfect elimination order $2,\dots,n$, by using induction on $n$, $J_c(H)$ has the desired linear quotients order. We distinguish the two possible cases.\smallskip
		
		\textbf{Case 1.} Suppose $v_i=x_1x_rx_s\in x_1Q$. Let $u_j=x_1x_p\in x_1P$. Then $u_j:v_i=x_p$ is a variable. Now let $u_j=x_px_q\in I(H^c)$. If $u_j:v_i$ is a variable then we are done. Otherwise, $\{p,q\}\cap\{r,s\}=\emptyset$. Since $I(H^c)\subset P$ (see \cite[Lemma 2.4]{FShort}) we have $x_p\in P$ or $x_q\in P$. Say $x_p\in P$. Then $x_1x_p:x_1x_rx_s=x_p$ divides $u_j:v_i=x_px_q$, as desired. Now consider $v_j:v_i$ with $j<i$. Then $v_j=x_1x_{r'}x_{s'}\in x_1Q$ with $x_{r'}x_{s'}>_{\lex}x_rx_s$. If $v_j:v_i$ is a variable, we are done. Otherwise $v_j:v_i=x_{r'}x_{s'}$ with $r'<s'$ and $r'<r<s$. Since $1,2,\dots,n$ is a perfect elimination order of $G$, then $\{r,r',s,s'\}$ is a 4-clique of $G$. Hence $x_1x_{r'}x_{s}=v_h\in x_1Q$ and $v_h>_{\lex}v_i$. Finally, $v_h:v_i=x_{r'}$ divides $v_j:v_i$.\smallskip
		
		\textbf{Case 2.} Suppose $v_i=x_rx_sx_t\in K_3(H)$. Let $u_j=x_1x_p\in x_1P$. Then $x_1$ divides $u_j:v_i$. If one of the three variables $x_r,x_s,x_t$ belongs to $P$, say $x_r$ then $x_1x_r=u_h$ for some $h$ and $u_h:v_i=x_1$ divides $u_j:v_i$. Suppose now that none of the variables $x_r,x_s,x_t$ belong to $P$. Then $\{1,r\},\{1,s\},\{1,t\}\in E(G)$, and $\{1,r,s,t\}$ is a $4$-clique of $G$. Hence $x_1x_rx_s=v_h\in x_1Q$ for some $h<i$, and $v_h:v_i=x_1$ divides $u_j:v_i$. Now, let $u_j=x_px_q\in I(H^c)$. In this case, since by induction $J_c(H)$ has the desired linear quotients order, it follows that $u_j:v_i$ is divided by some variable which is equal to $w:v_i$ for some $w\in\{u_1,\dots,u_m,v_1,\dots,v_{i-1}\}\cap\mathcal{G}(J_c(H))$. Now, let $v_j\in K_3(G)$ with $j<i$. If $v_j\in x_1Q$, then $x_1$ divides $v_j:v_i$. In this case, it is enough to consider any generator $u_h:v_i$ with $u_h\in x_1P$ and apply the previous arguments. If $v_j\in K_3(H)$ we use again that $J_c(H)$ has the desired linear quotients order by induction.\smallskip
		
		(c) $\Rightarrow$ (b) It follows from \cite[Theorem 8.2.15]{HHBook}.
	\end{proof}
	
	Due to some experimental evidence, we ask the following question.
	\begin{Question}\label{que:Jc}
		Suppose that $G$ is a chordal graph. Is it true that $J_c(G)^k$ is componentwise linear, or even has linear quotients, for all $k\ge1$ ?
	\end{Question}
	
	Notice that $J_c(G)=K_2(G^c)+K_3(G)$. For an integer $t\ge2$, and a finite simple graph $G$, one can consider more generally the ideal $J=K_t(G^c)+K_{t+1}(G)$.
	\begin{Question}
		For which graphs $G$ is the ideal $J=K_t(G^c)+K_{t+1}(G)$ componentwise linear, or has linear quotients? When this is the case, is it true that $J^k$ is componentwise linear, or has linear quotients, for all $k\ge1$ ?
	\end{Question}
	
	As a consequence of Theorem \ref{thm:Ic(G)sCM} we have
	\begin{Theorem}\label{thm:Ic(G)CM}
		Let $G$ be a finite simple graph without isolated vertices. The following conditions are equivalent. 
		\begin{enumerate}
			\item[\textup{(a)}] $I_c(G)$ is Cohen-Macaulay.
			\item[\textup{(b)}] $J_c(G)$ has linear resolution.
			\item[\textup{(c)}] $G$ is the complete graph on $[n]$ or a forest.
		\end{enumerate}
	\end{Theorem}
	\begin{proof}
		(a) $\Leftrightarrow$ (b) follows by the Eagon-Reiner theorem \cite[Theorem 8.1.9]{HHBook}. Now, (b) holds if and only if $J_c(G)$ is componentwise linear and generated in a single degree. By Theorem \ref{thm:Ic(G)sCM} and Corollary \ref{cor:J_c(G)dual}, it follows that (b) holds if and only if $G$ is chordal and either $J_c(G)=I(G^c)$ or $J_c(G)=K_3(G)$, that is, either $G$ is $K_3$-free, which means that $G$ is a forest, or $G$ is the complete graph, as desired. 
	\end{proof}
	
	It would be of interest to answer the following question.
	\begin{Question}
		When does $I_c(G)$ satisfy Serre's condition $(S_2)$ ?
	\end{Question}
	
	Let $(R,\m,K)$ be either a local ring or a standard graded $K$-algebra, with (graded) maximal ideal $\m$, which is Cohen-Macaulay and admits a canonical module $\omega_R$. The \textit{canonical trace} of $R$ is defined as the ideal
	$$
	\textup{tr}(\omega_R)=\sum_{\varphi\in\textup{Hom}_R(\omega_R,R)}\varphi(\omega_R).
	$$
	
	Following \cite{HHS}, we say that $R$ is \textit{nearly Gorenstein} if $\m\subseteq\textup{tr}(\omega_R)$. Since $R$ is Gorenstein if and only if $\textup{tr}(\omega_R)=R$, Gorenstein rings are nearly Gorenstein.
	
	As before, we say that $I\subset S$ is (nearly) Gorenstein if $S/I$ is a (nearly) Gorenstein ring. We close this section by classifying when $I_c(G)$ is nearly Gorenstein.
	
	\begin{Lemma}\label{lem:IndGor}
		The following statements hold.
		\begin{enumerate}
			\item[\textup{(a)}] $I_{n,d}$ is Gorenstein if and only if $d\in\{0,1,n\}$.
			\item[\textup{(b)}] $I_{n,d}$ is nearly Gorenstein if and only if $d\in\{0,1,2,n\}$.
		\end{enumerate}
	\end{Lemma}
	\begin{proof}
		(a) If $d=0$, then $I_{n,0}=(0)$ is Gorenstein.
		
		Let $d>0$. By \cite[Lemma 2 and Corollary 3]{CF2024}, $\pd\,I_{n,d}=n-d$ and $I_{n,d}$ is Cohen-Macaulay. By \cite[Corollary 7.4.2]{HHBook}, the last non-zero Betti number of $I_{n,d}$ is
		$$
		\beta_{n-d}(I_{n,d})\ =\sum_{\substack{u\in\mathcal{G}(I_{n,d})\\ n\in\supp(u)}}\binom{n-d}{n-d}\ =\ |\{u\in\mathcal{G}(I_{n,d})\ :\ x_n\ \textup{divides}\ u\}|,
		$$
		which is equal to 1 if and only if $d=1$ or $d=n$. The assertion follows.\smallskip
		
		(b) By part (a), we can assume that $2\le d\le n-1$. Notice that $I_{n,d}=I_\Delta$ can be regarded as the Stanley-Reisner ideal of the simplicial complex $\Delta$ with set of facets $\mathcal{F}(\Delta)=\{\{i_1,\dots,i_{d-1}\}:\ 1\le i_1<\dots<i_{d-1}\le n\}$. By \cite[Theorem B(Y)(4)]{MV}, it follows that $I_{n,d}$ is nearly Gorenstein, in the given range $2\le d\le n-1$, only when $d=2$ and this completes the proof.
	\end{proof}
	
	\begin{Corollary}\label{cor:Ic(G)Gor}
		Let $G$ be a finite simple graph without isolated vertices. The following conditions are equivalent.
		\begin{enumerate}
			\item[\textup{(a)}] $I_c(G)$ is nearly Gorenstein.
			\item[\textup{(b)}] $G\in\{K_2,\,K_3,\,2K_2,\,K_4,\,P_3,\,P_4\}$.
		\end{enumerate}
	\end{Corollary}
	\begin{proof}
		Let $G$ be nearly Gorenstein. Then by Theorem \ref{thm:Ic(G)CM}, $G$ is either a complete graph or a forest.
		
		If $G=K_n$ is the complete graph, then $I_c(G)=I_{n,n-2}$ and using Lemma \ref{lem:IndGor}, we have that $I_c(G)$ is nearly Gorenstein if and only if $n-2\in\{0,1,2,n\}$. Hence, if and only if, $n\in\{2,3,4\}$, that is $G\in\{K_2,K_3,K_4\}$.
		
		If $G$ is a forest, by Corollary \ref{cor:Ic(G)dim}, $I_c(G)$ is a Cohen-Macaulay ideal of height 2. Now, we apply \cite[Theorem 4.1]{F25}, (see also \cite[Corollary 3.6]{FHST}). According to this result, the only nearly Gorenstein squarefree monomial ideals in $S$ of height 2 are
		\begin{enumerate}
			\item[(i)] $I=(u,v)$, where $u,v$ are squarefree monomials with $\supp(u)\cap\supp(v)=\emptyset$,
			\item[(ii)] $I=(x_1x_2,x_1x_3,x_2x_3)$ with $n=3$,
			\item[(iii)] $I=(x_1x_3,x_1x_4,x_2x_4)$ with $n=4$.
		\end{enumerate}
		In case (i), since $G$ does not have isolated vertices, it follows that either $G=2K_2$ (in which case $I_c(2K_2)=(x_1x_2,x_3x_4)$) or $G=P_3$ (in which case $I_c(P_3)=(x_1,x_3)$). The case (ii) does not occur, because $\supp(I)=[3]$ but $I_c(G)$ should be generated in degree $|\supp(I)|-2=1$. Finally, in case (iii) we have $G=P_4$.
	\end{proof}
	
	As a consequence, we have:
	
	\begin{Theorem}\label{thm:IcGor}
		Let $G$ be a finite simple graph without isolated vertices. The following conditions are equivalent.
		\begin{enumerate}
			\item[\textup{(a)}] $I_c(G)$ is Gorenstein.
			\item[\textup{(b)}] $G\in\{K_2,\,K_3,\,2K_2,\,P_3\}$.
		\end{enumerate}
	\end{Theorem}
	\begin{proof}
		$I_c(G)$ is nearly Gorenstein if and only if $G\in\{K_2,K_3,2K_2,K_4,P_3,P_4\}$. The ideals $I_c(K_2)=I_{2,0}=(0)$, $I_c(K_3)=I_{3,1}=(x_1,x_2,x_3)$, $I_c(2K_2)=(x_1x_2,x_3x_4)$ and $I_c(P_3)=(x_1,x_3)$ are clearly Gorenstein. On the other hand, $I_c(K_4)=I_{4,2}$ is not Gorenstein by Lemma \ref{lem:IndGor}(a), and the ideal $I_c(P_4)=(x_1x_3,x_1x_4,x_2x_4)$ is not Gorenstein for it has Cohen-Macaulay type equal to $|\mathcal{G}(I_c(P_4))|-1=2$.
	\end{proof}
	
	\section{Polymatroidal property of $I_c(G)$}\label{sec3}
	
	In this short section we recover and strengthen a result due to, independently, Blum \cite[Corollary 4.3]{B} Ohsugi and Hibi \cite[Theorem 1.1 and Corollary 1.3]{OH}, and Nasernejad, Khashyarmanesh and Qureshi \cite[Theorem 2.3]{KNQ}.
	
	Recall that $G$ is called a \textit{complete multipartite graph} if there exists a partition of the vertex set $V(G)=A_1\sqcup\cdots\sqcup A_m$, with $m\ge2$ and $A_i\ne\emptyset$ for $i=1,\dots,m$, such that $E(G)=\{\{r,s\}:\ r\in A_i,\,s\in A_j, i\neq j\}$.  
	
	\begin{Theorem}\label{thm:Ic(G)polym}
		Let $G$ be a finite simple graph without isolated vertices. The following conditions are equivalent.
		\begin{enumerate}
			\item[\textup{(a)}] $I(G)$ is matroidal.
			\item[\textup{(b)}] $I_c(G)$ is matroidal.
			\item[\textup{(c)}] $G^c$ is a $P_3$-free graph.
			\item[\textup{(d)}] $G^c$ is a disjoint union of cliques.
			\item[\textup{(e)}] $G$ is a complete multipartite graph.
		\end{enumerate}
	\end{Theorem}
	\begin{proof}
		The equivalence between (a) and (b) follows from Theorem \ref{thm:Icpolym}.\smallskip
		
		(b) $\Rightarrow$ (c) Assume that condition (b) holds, and suppose by contradiction that (c) does not hold. Then $G^c$ contains an induced path of length three, say with the vertex set $\{i,p,q\}$ and the edges $\{p,i\},\{i,q\}\in E(G^c)$. Then $e=\{p,q\}\in E(G)$ and $i$ is not adjacent to $p$ and $q$ in $G$. Since $G$ does not contain isolated vertices, we can find an edge $f=\{i,r\}\in E(G)$ with $r\ne p,q$. Now, let $u={\bf x}_{[n]\setminus e}$ and $v={\bf x}_{[n]\setminus f}$. We have $\deg_{x_i}(u)>\deg_{x_i}(v)$. Therefore, by the exchange property there is $j$ with $\deg_{x_j}(u)<\deg_{x_j}(v)$ such that $x_j(u/x_i)\in\mathcal{G}(I_c(G))$. This is impossible, because $\{i,p\},\{i,q\}\notin E(G)$. The assertion follows.\smallskip
		
		(c) $\Rightarrow$ (b) We prove that $I_c(G)$ satisfies the exchange property. To this end, let $u,v\in\mathcal{G}(I_c(G))$ with $\deg_{x_i}(u)>\deg_{x_i}(v)$. Then $u={\bf x}_{[n]\setminus e}$ and $v={\bf x}_{[n]\setminus f}$ for some edges $e,f\in E(G)$, with $i\in f$ and $i\notin e$. Let $e=\{p,q\}$. Then $p,q\ne i$. Notice that if $\deg_{x_j}(u)<\deg_{x_j}(v)$ for some integer $j$, then $j\in\{p,q\}$. For the exchange property to hold, we should have either $x_q(u/x_i)={\bf x}_{[n]\setminus\{i,p\}}\in\mathcal{G}(I_c(G))$ or $x_p(u/x_i)={\bf x}_{[n]\setminus\{i,q\}}\in\mathcal{G}(I_c(G))$. That is, either $\{i,p\}\in E(G)$ or $\{i,q\}\in E(G)$. Suppose this was not the case. Then $\{p,i\},\{i,q\}\in E(G^c)$ and $\{p,q\}\notin E(G^c)$. Hence $G^c$ would contain an induced path of length three, a contradiction. This proves our assertion.\smallskip
		
		Finally, it is a simple exercise in graph theory to show that (c) $\Leftrightarrow$ (d) $\Leftrightarrow$ (e).
	\end{proof}
	
	\section{Depth and regularity functions of $I_c(G)$}\label{sec4}
	
	For a finite simple graph $G$ we denote by $c(G)$ the number of connected components of $G$ which are not isolated vertices. 
	\begin{Theorem}\label{thm:Ic(G)regdepth}
		Let $G$ be a finite simple graph on the vertex set $[n]$. Then        \smallskip
		
		$$
		\reg\,I_c(G)^k\ =\ \begin{cases}
			(|V(G)|-1)k&\textit{if}\ \ 1\le k\le c(G)-2,\\
			(|V(G)|-2)k+c(G)-1&\textit{if}\ \ k\ge c(G)-1,
		\end{cases}
		$$
		and
		$$
		\depth\,S/I_c(G)^k\,\ge\,\depth\,S/I_c(G)^{k+1},
		$$
		for all $k\ge1$.
	\end{Theorem}
	
	For the proof of this result we recall  few facts. Let $I,I_1,I_2\subset S$ be monomial ideals such that $\mathcal{G}(I)=\mathcal{G}(I_1)\sqcup\mathcal{G}(I_2)$ is the disjoint union of $\mathcal{G}(I_1)$ and $\mathcal{G}(I_2)$. Following \cite[Definition 1.1]{FHT}, we say that $I=I_1+I_2$ is a \textit{Betti splitting} if
	\begin{equation}\label{eq:BettiSplitEq}
		\beta_{i,j}(I)=\beta_{i,j}(I_1)+\beta_{i,j}(I_2)+\beta_{i-1,j}(I_1\cap I_2),\ \ \ \text{for all}\ i,j\ge0.
	\end{equation}
	In this case, by \cite[Corollary 2.2]{FHT} we have
	\begin{align*}
		\depth\,S/I\ &=\ \min\{\depth\,S/I_1,\,\depth\,S/I_2,\,\depth\,S/(I_1\cap I_2)-1\},\\
		\reg\,I\ &=\ \max\{\reg\,I_1,\,\reg\,I_2,\,\reg\,I_1\cap I_2-1\}.
	\end{align*}
	
	Consider the natural short exact sequence $0\rightarrow I_1\cap I_2\rightarrow I_1\oplus I_2\rightarrow I\rightarrow0$. By \cite[Proposition 2.1]{FHT}, equation (\ref{eq:BettiSplitEq}) holds, that is $I=I_1+I_2$ is a Betti splitting, if and only if, the induced maps  
	$$
	\Tor_i^S(K,I_1\cap I_2)\rightarrow\Tor_i^S(K,I_1)\oplus\Tor_i^S(K,I_2)
	$$
	in $\Tor$ of the above sequence
	are zero for all $i\ge0$. If this is the case, we say that the inclusion maps $I_1\cap I_2\rightarrow I_1$ and $I_1\cap I_2\rightarrow I_2$ are \textit{$\Tor$-vanishing}.
	
	We quote \cite[Proposition 3.1]{EK} (see also \cite[Lemma 4.2]{NV}).
	
	\begin{Lemma}\label{lemma:criteria-bs}
		Let $J,L\subset S$ be non-zero monomial ideals with $J\subset L$. Suppose there exists a map $\varphi:\mathcal{G}(J)\rightarrow\mathcal{G}(L)$ such that for any $\emptyset\ne\Omega\subseteq \mathcal{G}(J)$ we have
		$$
		\lcm(u:u\in\Omega)\in \mathfrak{m}\cdot(\lcm(\varphi(u):u\in\Omega)),
		$$
		where $\mathfrak{m}=(x_1,\dots,x_n)$. Then the inclusion map $J\rightarrow L$ is $\Tor$-vanishing.
	\end{Lemma}
	
	Let $I\subset S$ be a monomial ideal. We denote by $\partial I$ the ideal generated by the elements $u/x_i$, with $u\in\mathcal{G}(I)$ and $i\in\supp(u)$.\smallskip
	
	The next result essentially follows from \cite[Proposition 4.4]{NV} due to Nguyen and Vu. See also \cite[Lemma 1.3 and Theorem 1.2]{CFL}.
	
	\begin{Lemma}\label{lemma:bs}
		Let $J,L\subset S$ be non-zero monomial ideals. Suppose that $\partial J\subset L$. Then $J\subseteq\mathfrak{m}L$ and there exists a map $\varphi:\mathcal{G}(J)\rightarrow\mathcal{G}(L)$ satisfying the assumptions of Lemma \ref{lemma:criteria-bs}. In particular, the inclusion map $J\rightarrow L$ is $\Tor$-vanishing.
	\end{Lemma}
	
	The next two basic remarks will be used without reference in what follows.
	\begin{Remark}
		Let $I\subset S$ be a graded ideal and let $f\in S$ be a homogeneous element. Then $\reg\,fI=\reg\,I+\deg(f)$ and $\depth\,S/(fI)=\depth\,S/I$.
	\end{Remark}
	\begin{Remark}
		Let $I_1 \subset S_1=K[x_1,\ldots,x_n]$ and $I_2 \subset S_2=K[y_1,\ldots,y_m]$ be proper monomial ideals in polynomial rings over disjoint sets of variables. Then, we have $\reg\,I_1I_2=\reg\,I_1+\reg\,I_2$ and $\depth\,S/(I_1I_2)=\depth\,S_1/I_1+\depth\,S_2/I_2$. Moreover, if $I_2=(1)$ but $I_1\ne(1)$, then $\depth\,S/(I_1I_2)=\depth\,S_1/I_1+m$. Similarly, if $I_1=(1)$ but $I_2\ne(1)$, then $\depth\,S/(I_1I_2)=\depth\,S_2/I_2+n$.
	\end{Remark}
	
	Now, we are ready to prove Theorem \ref{thm:Ic(G)regdepth}.
	\begin{proof}[Proof of Theorem \ref{thm:Ic(G)regdepth}]
		We proceed by induction on $c=c(G)\ge1$. If $c=1$, it follows by \cite[Theorem B]{F25b} that $I_c(G)^k$ has linear resolution for all $k\ge1$. That is, $\reg\,I_c(G)^k=(|V(G)|-2)k$ for all $k\ge1$. Furthermore, by \cite[Proposition 10.3.4]{HHBook} it follows that the depth function $k\mapsto\depth\,S/I_c(G)^k$ is non-increasing.
		
		Now, let $c>1$ and write $G=G_1\sqcup G_2$, with $c(G_1)=c-1$ and $c(G_2)=1$ such that $G_1$ contains no isolated vertices of $G$. Identifying the variables of $S$ with the vertices of $G$, after a suitable renaming, we may assume that $V(G_1)=\{x_1,\dots,x_n\}$ and $V(G_2)=\{y_1,\dots,y_m\}$. Let $$I_1\ =\ ({\bf x}_{[n]}/(x_ix_j)\ :\ \{x_i,x_j\}\in E(G_1))\subset S$$ be the complementary edge ideal of $G_1$ with respect to the set of variables $\{x_1,\dots,x_n\}$. Similarly, let $I_2=({\bf y}_{[m]}/(y_iy_j)\ :\ \{y_i,y_j\}\in E(G_2))\subset S$ be the complementary edge ideal of $G_2$ with respect to the set of variables $\{y_1,\dots,y_m\}$. Since $c(G_1),c(G_2)<c$, by induction both ideals $I_1$ and $I_2$ satisfy the statement.
		
		Let $I=I_c(G)$ and notice that $I={\bf y}_{[m]}I_1+{\bf x}_{[n]}I_2$. Hence,
		$$
		I^k\ = \ \sum_{h=0}^k{\bf y}_{[m]}^{k-h}{\bf x}_{[n]}^hI_1^{k-h}I_2^h
		$$
		for all $k\ge1$. Fix an integer $k\ge1$, and set $J_\ell=\sum_{h=0}^\ell{\bf y}_{[m]}^{k-h}{\bf x}_{[n]}^hI_1^{k-h}I_2^h$ for $0\le\ell\le k$.
		
		Notice that $J_k=I^k$. For each $0<\ell\le k$, we claim that
		\begin{equation}\label{eq:Iell-bs}
			J_\ell\ =\ J_{\ell-1}+{\bf y}_{[m]}^{k-\ell}{\bf x}_{[n]}^\ell I_1^{k-\ell}I_2^\ell
		\end{equation}
		is a Betti splitting of $J_\ell$. First, we notice that
		$$
		\mathcal{G}(J_\ell)\ =\ \mathcal{G}(J_{\ell-1})\sqcup\mathcal{G}({\bf y}_{[m]}^{k-\ell}{\bf x}_{[n]}^\ell I_1^{k-\ell}I_2^\ell).
		$$
		
		To this end, just observe that any $u\in\mathcal{G}(J_{\ell-1})$ and $v\in\mathcal{G}({\bf y}_{[m]}^{k-\ell}{\bf x}_{[n]}^\ell I_1^{k-\ell}I_2^\ell)$ have the same degree $(n+m-2)k$, but they are different because
		$$
		\sum_{i=1}^n\deg_{x_i}(u)=h n+(n-2)(k-h)=(n-2)k+2h< (n-2)k+2\ell=\sum_{i=1}^n\deg_{x_i}(v),
		$$
		for some $0\le h<\ell$. Now, we compute the intersection
		\begin{align*}
			L\ =\ J_{\ell-1}\cap({\bf y}_{[m]}^{k-\ell}{\bf x}_{[n]}^\ell I_1^{k-\ell}I_2^\ell)\ &=\ (\sum_{h=0}^{\ell-1}{\bf y}_{[m]}^{k-h}{\bf x}_{[n]}^hI_1^{k-h}I_2^h)\cap({\bf y}_{[m]}^{k-\ell}{\bf x}_{[n]}^\ell I_1^{k-\ell}I_2^\ell)\\
			&=\phantom{\ (}\sum_{h=0}^{\ell-1}[({\bf y}_{[m]}^{k-h}{\bf x}_{[n]}^hI_1^{k-h}I_2^h)\cap({\bf y}_{[m]}^{k-\ell}{\bf x}_{[n]}^\ell I_1^{k-\ell}I_2^\ell)]\\
			&=\phantom{\ (}\sum_{h=0}^{\ell-1}[({\bf y}_{[m]}^{k-h}I_2^h)\cap({\bf y}_{[m]}^{k-\ell}I_2^\ell)][({\bf x}_{[n]}^hI_1^{k-h})\cap({\bf x}_{[n]}^\ell I_1^{k-\ell})].
		\end{align*}
		
		Notice that $({\bf x}_{[n]})\subset I_1$ and $({\bf y}_{[m]})\subset I_2$. Now, for any $h<\ell$, we have $k-h>k-\ell$ and so ${\bf y}_{[m]}^{k-h}I_2^h={\bf y}_{[m]}^{k-\ell}{\bf y}_{[m]}^{\ell-h}I_2^h\subset{\bf y}_{[m]}^{k-\ell}I_2^{\ell-h}I_2^h={\bf y}_{[m]}^{k-\ell}I_2^\ell$. Similarly, we have ${\bf x}_{[n]}^\ell I_1^{k-\ell}\subset{\bf x}_{[n]}^h I_1^{k-h}$. Consequently,
		\begin{align*}
			L\ =\ J_{\ell-1}\cap({\bf y}_{[m]}^{k-\ell}{\bf x}_{[n]}^\ell I_1^{k-\ell}I_2^\ell)\ &=\  \sum_{h=0}^{\ell-1}({\bf y}_{[m]}^{k-h}I_2^h{\bf x}_{[n]}^\ell I_1^{k-\ell})\ =\, (\sum_{h=0}^{\ell-1}{\bf y}_{[m]}^{k-h}I_2^h)({\bf x}_{[n]}^\ell I_1^{k-\ell})\\[3pt]
			&=\ {\bf y}_{[m]}^{k-\ell+1}I_2^{\ell-1}{\bf x}_{[n]}^\ell I_1^{k-\ell},
		\end{align*}
		where we used that $({\bf y}_{[m]}^k)\subset{\bf y}_{[m]}^{k-1}I_2\subset\cdots\subset{\bf y}_{[m]}^{k-(\ell-1)}I_2^{\ell-1}$.
		
		To conclude that (\ref{eq:Iell-bs}) is indeed a Betti splitting we must prove that the inclusion maps $L\rightarrow J_{\ell-1}$ and $L\rightarrow{\bf y}_{[m]}^{k-\ell}{\bf x}_{[n]}^\ell I_1^{k-\ell}I_2^\ell$ are $\Tor$-vanishing.\smallskip
		
		We show that the first map is $\Tor$-vanishing, the other case is similar. We claim that $\partial({\bf x}_{[n]}^\ell I_1^{k-\ell})\subset{\bf x}_{[n]}^{\ell-1} I_1^{k+1-\ell}$. Indeed, take $u={\bf x}_{[n]}^\ell u_1\cdots u_{k-\ell}\in\mathcal{G}({\bf x}_{[n]}^\ell I_1^{k-\ell})$ with $u_s\in\mathcal{G}(I_1)$ for all $s=1,\dots,k-\ell$, and let $x_i$ divide $u$. Since $G_1$ does not contain isolated vertices, there exists $x_j\in V(G_1)$ for which $\{x_i,x_j\}\in E(G_1)$. Then
		$$
		{\bf x}_{[n]}^{\ell-1}u_1\cdots u_{k-\ell}({\bf x}_{[n]}/(x_ix_j))\ \ \textup{divides}\ \ u/x_i.
		$$
		
		This shows that $u/x_i\in {\bf x}_{[n]}^{\ell-1} I_1^{k+1-\ell}$. Hence $\partial({\bf x}_{[n]}^\ell I_1^{k-\ell})\subset{\bf x}_{[n]}^{\ell-1} I_1^{k+1-\ell}$. Lemma \ref{lemma:bs} implies that there exists a map $\varphi:\mathcal{G}({\bf x}_{[n]}^\ell I_1^{k-\ell})\rightarrow\mathcal{G}({\bf x}_{[n]}^{\ell-1} I_1^{k+1-\ell})$ such that for any non-empty subset $\Omega\subset\mathcal{G}({\bf x}_{[n]}^\ell I_1^{k-\ell})$ we have
		\begin{equation}\label{eq:lcm}
			\lcm(u:u\in\Omega)\ \in\ \n\cdot(\lcm(\varphi(u):u\in\Omega)),
		\end{equation}
		where $\n=(x_1,\dots,x_n)$. Next, we define a map
		$$
		\Phi:\mathcal{G}(L)\rightarrow\mathcal{G}(J_{\ell-1})
		$$
		using the map $\varphi$. Let $u\in\mathcal{G}(L)$. Then $u=u_xu_y$, where $u_x\in\mathcal{G}({\bf x}_{[n]}^\ell I_1^{k-\ell})$ and $u_y\in\mathcal{G}({\bf y}_{[m]}^{k+1-\ell}I_2^{\ell-1})$. We put $\Phi(u)=\varphi(u_x)u_y$. Notice that $\Phi$ is well defined, because $\Phi(u)\in\mathcal{G}(({\bf x}_{[n]}^{\ell-1} I_1^{k+1-\ell})({\bf y}_{[m]}^{k+1-\ell}I_2^{\ell-1}))\subset\mathcal{G}(J_{\ell-1})$. Finally, let $\Omega\subset\mathcal{G}(L)$ be non-empty. Then $\lcm(u:u\in\Omega)= \lcm(u_x:u\in\Omega)\cdot\lcm(u_y:u\in\Omega)$ and by using (\ref{eq:lcm}) we have
		\begin{align*}
			\lcm(u:u\in\Omega)\ \in&\ \n\cdot(\lcm(\varphi(u_x):u\in\Omega))\cdot\lcm(u_y:u\in\Omega)\\
			=&\ \n\cdot(\lcm(\varphi(u_x)u_y:u\in\Omega))\\
			\subset&\ \m\cdot(\lcm(\Phi(u):u\in\Omega)),
		\end{align*}
		where $\m=(x_1,\dots,x_n,y_1,\dots,y_m)$. Lemma \ref{lemma:criteria-bs} implies that the inclusion $L\rightarrow J_{\ell-1}$ is indeed $\Tor$-vanishing, as wanted.\smallskip
		
		Let $S_1=K[x_1,\dots,x_n]$ and $S_2=K[y_1,\dots,y_m]$. Next, we claim that
		\begin{align}
			\label{eq:regbs}\reg\,I^k=\max&\left\{\substack{\displaystyle\reg\,I_1^k+mk,(n+m-2)k+1,\\[5pt]\displaystyle\max_{0<h<k}\{\reg\,I_1^{k-h}\!+mk+(n-2)h+1\}}\right\},\\[4pt]
			\label{eq:depthbs}\depth\frac{S}{I^k}=\min&\left\{\substack{\displaystyle\!\depth\,S_1/I_1^k\!+m,\depth\,S_2/I_2^k\!+n,\depth\,S_2/I_2^{k-1}\!+n-1,\\[5pt]\displaystyle\min_{0<h<k}\{\depth\,S_1/I_1^{k-h}+\depth\,S_2/I_2^h\}}\right\}.
		\end{align}
		
		Since (\ref{eq:Iell-bs}) is a Betti splitting and $I^k=J_k$, we have
		\begin{align*}
			\reg\,I^k&=\max\{\reg\,J_{k-1},\,\reg({\bf x}_{[n]}^kI_2^k),\,\reg({\bf y}_{[m]}{\bf x}_{[n]}^kI_2^{k-1})-1\},\\
			\depth\,S/I^k&=\min\{\depth\,S/J_{k-1},\,\depth\,S/({\bf x}_{[n]}^kI_2^k),\,\depth\,S/({\bf y}_{[m]}{\bf x}_{[n]}^kI_2^{k-1})-1\}.
		\end{align*}
		
		Since $c(G_2)=1$, we have $\reg({\bf x}_{[n]}^kI_2^k)=(n+m-2)k$ for all $k\ge1$. Analogously, $\reg({\bf y}_{[m]}{\bf x}_{[n]}^kI_2^{k-1})-1=m+nk+(m-2)(k-1)-1=(n+m-2)k+1$. Similar computations can be performed for the depth. Hence,
		\begin{align*}
			\reg\,I^k&=\max\{\reg\,J_{k-1},\,(n+m-2)k+1\},\\
			\depth\,S/I^k&=\min\{\depth\,S/J_{k-1},\,\depth\,S_2/I_2^k+n,\,\depth\,S/I_2^{k-1}+n-1\}.
		\end{align*}
		
		Iterating these computations to $J_{k-1},\dots,J_0$ by using the Betti splittings (\ref{eq:Iell-bs}), we see that the formulas (\ref{eq:regbs}) and (\ref{eq:depthbs}) indeed hold.
		
		By induction, the functions $k\mapsto\depth\,S_1/I_1^k$, $k\mapsto\depth\,S_2/I_2^k$ are non-increasing. It follows at once from   formula (\ref{eq:depthbs}) that $k\mapsto\depth\,S/I^k$ is non-increasing too.
		
		Now, since $I_1=I_c(G_1)$, $|V(G_1)|=n$ and $c(G_1)=c-1$, by induction we have
		\begin{equation}\label{IndI1}
			\reg\,I_1^k\ =\ \begin{cases}
				(n-1)k&\textit{if}\ \ 1\le k\le c-3,\\
				(n-2)k+c-2&\textit{if}\ \ k\ge c-2.
			\end{cases} 
		\end{equation}
		
		Let $1\le k\le c-2$. Then by (\ref{IndI1}) it follows that $\reg\,I_1^k=(n-1)k$, since for $k=c-2$ we have $\reg\,I_1^{c-2}=(n-2)(c-2)+(c-2)=(c-2)(n-1)=(n-1)k$. Combining the above formula with formula (\ref{eq:regbs}) we obtain
		\begin{align*}
			\reg\,I^k\ &=\ \max\left\{\substack{\displaystyle(n+m-1)k,\,(n+m-2)k+1,\\[5pt]\displaystyle\max_{0<h<k}\{(n-1)(k-h)+mk+(n-2)h+1\}}\right\}\\
			&=\ \max\left\{\substack{\displaystyle(n+m-1)k,\\[5pt]\displaystyle\max_{0<h<k}\{(n+m-1)k-h+1\}}\right\}\\
			&=\ (n+m-1)k=(|V(G)|-1)k.
		\end{align*}
		
		Now, let $k\ge c-1$. Fix $0<h<k$. Then $k-h\ge c-2$ for $0<h\le k+2-c$ and $k-h\le c-3$ for $h\ge k+3-c$. Since $k\ge c-1$, then $k+2-c\ge1$. Hence, the formula for $\reg\,I_1^{k-h}$ combined with (\ref{eq:regbs}) yields
		\begin{align*}
			\reg\,I^k\ &=\ \max\left\{\substack{\displaystyle(n+m-2)k+c-2,\,(n+m-2)k+1,\\[10pt]\displaystyle\max_{0<h\le k+2-c}\{(n-2)(k-h)+c-2+mk+(n-2)h+1\},\\[3pt]\displaystyle\max_{k+3-c\le h<k}\{(n-1)(k-h)+mk+(n-2)h+1\}}\right\}\\
			&=\ \max\left\{\substack{\displaystyle(n+m-2)k+c-2,\\[3pt]\displaystyle(n+m-2)k+1,\,(n+m-2)k+c-1,\\[3pt]\displaystyle\max_{k+3-c\le h<k}\{(n+m-1)k-h+1\}}\right\}\\[3pt]
			&=\ (n+m-2)+c-1=(|V(G)|-2)k+c(G)-1,
		\end{align*}
		because $c\ge2$. This concludes the inductive proof.
	\end{proof}\medskip
	
	Let $I\subset S$ be a homogeneous ideal. It is known by \cite{CHT,K} that $\reg\,I^k$ is a linear function of the form $ak+b$ for some $a\ge1,b\ge0$ and all $k\gg0$. The smallest value $k_0>0$ for which $\reg\,I^k=ak+b$ for all $k\ge k_0$, is called the \textit{index of regularity stability} of $I$, and is denoted by $\textup{rstab}(I)$.
	\begin{Corollary}
		Let $G$ be a finite simple graph. Then
		$$
		\textup{rstab}(I_c(G))\ =\ \max\{1,c(G)-1\}.
		$$
	\end{Corollary}
	
	\begin{Corollary}\label{cor:Ic(G)Betti}
		Let $G$ be a finite simple graph. The graded Betti numbers of $I_c(G)^k$ do not depend on the ground field $K$, for all $k\ge1$.
	\end{Corollary}
	\begin{proof}
		If $c(G)=1$, by \cite[Theorem B]{F25b}, $I_c(G)^k$ has linear quotients for all $k\ge1$, and hence the graded Betti numbers of $I_c(G)^k$ do not depend on $K$. Now, let $c(G)>1$ and fix $k\ge1$. Using the notation in the proof of Theorem \ref{thm:Ic(G)regdepth}, by induction on $c$, we see that the graded Betti numbers of $J_0={\bf y}_{[m]}^{k}I_1^{k}$ do not depend on $K$. Assume inductively that the graded Betti numbers of $J_{\ell-1}$ do not depend on $K$. Since $$L=J_{\ell-1}\cap ({\bf y}_{[m]}^{k-\ell}{\bf x}_{[n]}^\ell I_1^{k-\ell}I_2^\ell)={\bf y}_{[m]}^{k-\ell+1}I_2^{\ell-1}{\bf x}_{[n]}^\ell I_1^{k-\ell},$$
		and since ${\bf y}_{[m]}^{k-\ell+1}I_2^{\ell-1}$ and ${\bf x}_{[n]}^\ell I_1^{k-\ell}$ live in polynomial rings in disjoint variables, by using \cite[Proposition 4.1]{HRR} and induction on $c$, then the graded Betti numbers of $L$ are independent of $K$. The Betti splitting (\ref{eq:Iell-bs}) 
		implies that the graded Betti numbers of
		$J_\ell$ do not depend on $K$. Since $J_k=I_c(G)^k$, this concludes the proof.
	\end{proof}
	
	\begin{Corollary}\label{cor:somepower}
		The following conditions are equivalent.
		\begin{enumerate}
			\item[\textup{(a)}] $I_c(G)^k$ has linear resolution, for some $k\ge1$ $($equivalently, for all $k\ge1)$.
			\item[\textup{(b)}] $I_c(G)^k$ has linear quotients, for some $k\ge1$ $($equivalently, for all $k\ge1)$.
			\item[\textup{(c)}] $c(G)=1$.
		\end{enumerate}
	\end{Corollary}
	\begin{proof}
		If (a) or (b) holds, then $\reg\,I_c(G)^{k}=(|V(G)|-2)k$ for some $k$. Theorem \ref{thm:Ic(G)regdepth} implies that $c(G)=1$ and so (c) holds. Conversely, it is shown in \cite[Theorem B]{F25b} that (c) implies (a) and (b).
	\end{proof}
	
	The \textit{girth} of a graph $G$, denoted by $\girth(G)$, is the minimum size of an induced cycle of $G$ if $G$ contains some cycles, otherwise, if $G$ is a forest we set $\girth(G)=\infty$.
	
	Recall that the $k$th symbolic power of a squarefree monomial ideal $I\subset S$ is defined as $I^{(k)}=\bigcap_{i=1}^mP_{F_i}^k$, where $I=\bigcap_{i=1}^mP_{F_i}$ is the minimal primary decomposition of $I$. We have $I^k\subset I^{(k)}$ for all $k\ge1$, but in general equality does not hold.
	
	To avoid unnecessary technicalities, in the next result, proved by Hoa and Trung \cite[Theorem 2.9]{HT}, and by Minh and Vu \cite[Theorem 1.1]{MV23}, we consider only graphs without isolated vertices and having at least two edges.
	\begin{Theorem}
		Let $G$ be a finite simple graph without isolated vertices and having at least two edges. If $\girth(G)\ge5$, then $\reg\,J_c(G)=3$ and $\reg\,J_c(G)^k=\reg\,J_c(G)^{(k)}=2k$ for all $k\ge2$. Otherwise,
		$$
		\reg\,J_c(G)^k\ =\ \reg\,J_c(G)^{(k)}\ =\ \begin{cases}
			3k&\textit{if}\ \girth(G)=3,\\
			2k+1&\textit{if}\ \girth(G)=4,\\
			2k&\textit{if}\ \girth(G)=\infty,
		\end{cases}
		$$
		for all $k\ge1$.
	\end{Theorem}
	
	In view of this result, we conclude the paper with two natural questions.
	\begin{Question}
		Are the functions $k\mapsto\!\depth\frac{S}{J_c(G)^k},\,k\mapsto\!\depth\frac{S}{J_c(G)^{(k)}}$ non-increasing?
	\end{Question}
	\begin{Question}
		Do we have $\reg\,I_c(G)^k=\reg\,I_c(G)^{(k)}$, for all $k\ge1$ ?
	\end{Question}\bigskip
	
	\textbf{Acknowledgment.}
	A. Ficarra was supported by the Grant JDC2023-051705-I funded by
	MICIU/AEI/10.13039/501100011033 and by the FSE+ and also by INDAM (Istituto Nazionale Di Alta Matematica). We would like to thank Adam Van Tuyl and Ayesha Asloob Qureshi for informing us about the papers \cite{ALS} and  \cite{HQS}, respectively.
	
\end{document}